\documentclass{article}
\usepackage{geometry}
\usepackage{graphicx}
\usepackage{caption}
\usepackage{subcaption}
\graphicspath{ {./sdraft/imggit/} }

\usepackage[T1]{fontenc}

\usepackage{amsmath, amsfonts, amsthm, amssymb}
\usepackage{xcolor}
\usepackage{mathtools}
\usepackage{hyperref}
\usepackage{nicefrac}
\usepackage{algorithm}
\usepackage{algpseudocode}
\usepackage{todonotes}
\usepackage{enumitem}
\usepackage{units}
\usepackage[nolabel, final]{showlabels}
\usepackage{authblk}

\newcommand{\hhmu}{h,\hat{\mu}}
\newcommand{\Hcal}{\mathcal{H}}

\newcommand{\bu}{\mathbf{b}}
\newcommand{\Th}{\mathcal{T}_{h}}

\newcommand{\aSUPG}{a_{\mathrm{SUPG}}}
\newcommand{\uhr}[1][]{u^{#1}_{h,\hat{\mu}}}

\newcommand{\lrp}[1]{ \left( #1 \right)}
\newcommand{\piun}[1][n]{\pi^{#1}u}

\newcommand{\trnorm}[1]{{\left\vert\kern-0.25ex\left\vert\kern-0.25ex\left\vert #1 
    \right\vert\kern-0.25ex\right\vert\kern-0.25ex\right\vert}}

\newcommand{\Lthm}{L^2_{\hat{\mu}}}

\newcommand{\hmu}{\hat{\mu}}

\newcommand{\dt}{\triangle t}

\newcommand{\vti}{\tilde{v}}
\newcommand{\eti}[1]{\tilde{e}^{n+1}}
\newcommand{\Clocbasinv}{C_{\mathrm{lbi}}}

\newtheorem{assumption}{Assumption}[section]
\newtheorem{lemma}{Lemma}[section]
\newtheorem{proposition}{Proposition}[section]
\newtheorem{theorem}{Theorem}[section]

\newtheorem{alg}{Algorithm}

\title{Error estimates for SUPG-stabilised Dynamical Low Rank Approximations}
\author[1]{Fabio Nobile}
\author[1]{Thomas Trigo Trindade}
\affil[1]{CSQI, \'Ecole Polytechnique F\'ed\'erale de Lausanne, Switzerland}

\begin{document}

\maketitle


\abstract{
	We perform an error analysis of a fully discretised Streamline Upwind Petrov Galerkin Dynamical Low Rank (SUPG-DLR) method for random time-dependent advection-dominated problems. The time integration scheme has a splitting-like nature, allowing for potentially efficient computations of the factors characterising the discretised random field. 
	The method allows to efficiently compute a low-rank approximation of the true solution, while naturally ``inbuilding'' the SUPG stabilisation. 
	Standard error rates in the $\|\cdot\|_{L^2}$ and $\|\cdot\|_{\mathrm{SUPG}}$-norms are recovered. 
	Numerical experiments validate the predicted rates. 
}

\section{Introduction}

The simulation of random time-dependent advection-dominated problems 
\begin{equation} \label{eqn:adv-diff-reac}
        \partial_t u  - \varepsilon \Delta u + \mathbf{b} \cdot \nabla u + c u = f, \quad \mathrm{in}\,D \subset  \mathbb{R}^d,
\end{equation}
with coefficients $\varepsilon$, $\bu$, $c$ and data $f$ depending on some random parameter $\omega \in \Omega$, with probability measure $\mu$ on $\Omega$, remains a challenge for multiple reasons. 
These processes often have poorly decaying Kolmogorov $n$-widths in the time-space domain, even if at each point in time the solution profile is well-approximated by a small subspace. 
Furthermore, it is well-known that applying the standard Finite Element Method to such problems causes the numerical solution to display unphysical spurious oscillations, in particular when the solution has sharp gradients and/or boundary layers.
For practical purposes, it becomes necessary to remove or alleviate these oscillations by using some stabilisation strategy.

The purpose of~\cite{notr23} was to introduce the generalised Petrov-Galerkin Dynamical Low Rank (PG-DLR) framework and its particularisation to the Streamline Upwind/Petrov-Galerkin (SUPG-DLR), which allows to simultaneously tackle both issues. 
The Dynamical Low Rank (DLR)~\cite{kolu07} framework, in this work written in its Dynamically Orthogonal (DO)~\cite{sale09} formalism, consists in seeking an approximation of the form $u_{\mathrm{DLR}} = \sum_{i=1}^R U_i(t,x) Y_i(t, \omega)$ of the solution $u_{\mathrm{true}}(t,x,\omega)$ of~\eqref{eqn:adv-diff-reac}. 
The peculiar feature of this framework is that the \textit{physical} $\{U_i(t,x)\}_{i=1}^R$ and the \textit{stochastic} modes $\{Y_i(t, \omega)\}_{i=1}^R$ evolve in time to follow a (quasi-)optimal low-rank approximation of $u_{\mathrm{true}}$, making it suited for the type of transport-dominated problems described above. 
As an extension of that framework, the PG-DLR framework allows to seamlessly import many stabilisation techniques that can be framed as generalised Petrov-Galerkin problems. 

The focus of this paper is an error analysis of the SUPG-DLR framework. 
This work inscribes itself within a growing body of literature addressing the stabilisation of Reduced Order Models, including e.g.~\cite{tobaro18,giiljowe15}
for SUPG-stabilised POD methods for advection-dominated problems. 
An error analysis for the SUPG-stabilised POD method was carried out in~\cite{jomono22} for time-dependent advection-diffusion-reaction problems. 
In the DO setting, a noteworthy alternative to our method is the stabilisation based on Shapiro filters in~\cite{fele18}, applied after each time step to smooth out the oscillations.

\section{Problem setting \& SUPG-DLR approximations}

Solutions to random PDEs are function-valued random variables. 
In this work, we consider the advection-diffusion-reaction problem~\ref{eqn:adv-diff-reac} with homogeneous Dirichlet boundary conditions $u = 0$ on $\partial D$ and initial condition $u_{\rvert{t=0}} = u_0 \in L^2_{\hmu}(H^1_0(D))$. The coefficients verify the following the Coefficient Assumptions (\hypertarget{coefa}{\textsc{CoefA}}):~$\varepsilon > 0$, $c \in L_{\hmu}^{\infty}(L^{\infty}(D))$ and $ c(x, \omega) \geq c_0 > 0$ for a.e. $x \in D, \forall \omega \in \hat{\Omega}$, $f \in L^2_{\hmu}(L^2(D))$, $\bu \in (L^{\infty}(D))^d$, $\mathrm{div}\;\bu(x) = 0$. Therefore the solution $u_{\mathrm{true}}(t,\cdot,\omega)$ belongs to $H^1_0(D)$ for (almost) every $t > 0$ and $\omega \in \Omega$. 
The probability space is discretised via a collocation method (e.g., the Monte Carlo method), yielding the collocation points $\hat{\Omega} \coloneq \{\omega_i\}_{i=1}^{N_C} \subset \Omega$ and a discrete measure $\hmu$. $L^2_{\hmu}(\hat{\Omega})$ denotes the space of random variables, with scalar product $\mathbb{E}_{\hmu}[YZ] = \sum_{i=1}^{N_C} m_i Y_i Z_i$, where $\{m_i\}_{i=1}^{N_C}$ are positive weights summing up to $1$, and $Y_i = Y(\omega_i)$, $Z_i = Z(\omega_i)$. 
The random solution $u_{\mathrm{true}}(t, \cdot,\cdot)$ satisfies for almost every $t$, $u \in L^2_{\hmu}(\hat{\Omega}, X) \coloneqq L^2_{\hmu}(X)$, where $X = H^1_0(D)$ (with standard $H^1_0$-scalar product) or $L^2(D)$. 
These Bochner spaces admit the scalar product~$(u,v)_{L^2_{\hmu}(X)} = \sum_{i=1}^{N_C} m_i \langle u(\omega_i), v(\omega_i) \rangle_{X}$.
Hereafter, we use the shorthand notation $(\cdot, \cdot)$ and $\|\cdot\|$ to denote the $L^2_{\hat{\mu}}(L^2(D))$ inner product and norm. 

We use the Finite Elements Method on a quasi-uniform mesh $\mathcal{T}_h$ with characteristic mesh size $h$, and consider the space of continuous piece-wise polynomials of degree $k$, $V_h \coloneq \mathbb{P}_k^C(\mathcal{T}_h) \subset H^1_0(D)$ 
where $k$ denotes the polynomial degree and $N_h \coloneqq |V_h|$. 
In this work, we will consider the \textit{advection-dominated regime} with the condition $\| \bu \|_{L^{\infty}} h > 2 \varepsilon$ assumed true hereafter.

The numerical approximation $\tilde{u}_{\hhmu}$ is sought in $V_h \otimes \Lthm$.
The inverse inequality from standard Finite Element theory can be extended to elements in $V_h \otimes L^2_{\hat{\mu}}$, yielding~$\|\nabla \tilde{u}_{\hhmu}\| \leq C_I h^{-1} \|\tilde{u}_{\hhmu}\|$ for some $C_I > 0$ and every $\tilde{u}_{\hhmu} \in V_h \otimes \Lthm$, as the inequality holds pointwise in $\omega$. 
For the same reasons, the standard Poincaré inequality can be extended to $V_h \otimes \Lthm$, yielding $\|\tilde{u}_{\hhmu}\| \leq C_P \| \nabla \tilde{u}_{\hhmu} \|$, where $C_P$ is the Poincaré constant.
Hereafter, to lighten the notation, $\tilde{u}  \equiv \tilde{u}_{\hhmu} \in V_h \otimes \Lthm $.

The DLR approximation belongs to the \textit{differential manifold of $R$-rank functions}, defined as
\vspace{-1em}
\begin{multline}
        \mathcal{M}_{R} = \{\tilde{u} \in V_h \otimes L^2_{\hat{\mu}}(\hat{\Omega}) : \tilde{u} = \sum_{i=1}^R U_i Y_i, \text{ s.t. } \mathbb{E}_{\hat{\mu}}[Y_i Y_j] = \delta_{ij}, 
          \\
          \{U_i\}_{i=1}^R \text{ lin. ind. } 
 \text{and } \{U_i\}_{i=1}^R \in V_h, \{Y_i\}_{i=1}^R \in L^2_{\hat{\mu}}(\hat{\Omega})\}.
\end{multline}
Each point $u \in \mathcal{M}_R$ can be equipped with a tangent space, spanned by tangent vectors $\delta u = \sum_{i=1}^R \delta u_i Y_i + U_i \delta y_i$, uniquely identified by imposing the \textit{Dual Dynamically Orthogonal} (Dual DO) condition~\cite{muno18}, $\mathbb{E}[Y_i \delta y_j] = 0$ for $ i,j = 1, \ldots,R$. This leads to the following characterisation  
\begin{multline}
	\mathcal{T}_u \mathcal{M}_{R} = \{\delta u = \sum_{i=1}^R \delta u_i Y_i + U_i \delta y_i, \text{ such that } \{\delta u_i\}_{i=1}^R \in V_h, \\
          \{\delta y_i\}_{i=1}^R \in L^2_{\hmu}(\hat{\Omega}), 
	\mathbb{E}_{\hmu} [\delta y_i Y_j] = 0, \, \forall 1 \leq i,j \leq R
	\}.
\end{multline}
Given $U = (U_1, \ldots, U_R)$ and $Y = (Y_1, \ldots, Y_R)$ s.t. $u = U Y^{\top}$, the tangent space at $u$ is denoted by $\mathcal{T}_{\!\!U Y^{\top}} \mathcal{M}_{R}$. 
Furthermore, for an $\Lthm$-orthonormal set $Y$, let $\mathcal{Y} \coloneqq \mathrm{span}(Y_1, \ldots, Y_R)$, and $\mathcal{P}_{\mathcal{Y}}[v] = \sum_{i=1}^R \mathbb{E}[vY_i]Y_i$ and $\mathcal{P}^{\perp}_{\mathcal{Y}}[v] = v - \mathcal{P}_{\mathcal{Y}}[v]$. 
%

To recover dynamic equations for the physical and stochastic modes, the idea is to project Equation~\eqref{eqn:adv-diff-reac} onto the tangent space $\mathcal{T}_{UY^{\top}} \mathcal{M}_R$ at each time instant. 
The SUPG-DLR framework proposes to solve the problem
\begin{multline} \label{eqn:PG-start}
        (\dot{u}_{\mathrm{DLR}}, \tilde{v} +\delta \bu \cdot \nabla \tilde{v}) + \aSUPG(u_{\mathrm{DLR}}, \tilde{v}) = (f, \tilde{v}+\delta \bu \cdot \nabla \tilde{v}).
        \\ 
        \forall \tilde{v} \in \mathcal{T}_{u_{\mathrm{DLR}}}\mathcal{M}_{R}, \mathrm{a.e.}\; t \in (0,T],
\end{multline}
with 
\begin{multline*} \label{eqn:aSUPGref}
        \aSUPG(\tilde{u}, \tilde{v}) = (\varepsilon \nabla \tilde{u}, \nabla \tilde{v}) + (\bu \cdot \nabla \tilde{u}, \tilde{v}) + (c\tilde{u}, \tilde{v}) 
        \\ 
        +  \sum_{K \in \Th} \delta_K(- \varepsilon \Delta \tilde{u} + \bu \cdot \nabla \tilde{u} + c\tilde{u}, \bu \cdot \nabla \tilde{v})_{K,\Lthm},
\end{multline*}
where $(\cdot, \cdot)_{K,\Lthm} \coloneqq (\cdot, \cdot)_{L^2_{\hat{\mu}}(L^2(K))}$. Hereafter, we use a uniform stabilisation parameter $\delta \equiv \delta_K$ for each $K \in \mathcal{T}_h$.
 
Particularising the conditions in~\cite{notr23} to our setting, if 
(\hyperlink{coefa}{\textsc{CoefA}}) and
        \begin{equation} \label{eqn:deltaK-coerc}
                \delta \leq 
                        \min_{K \in \Th}
                \left\{ 
                        \frac{1}{2 \|c\|_{L^{\infty}_{\hmu}(L^{\infty})}}, 
                        \frac{h_K^2}{2 \varepsilon C_I^2},
                        \frac{h_K}{\|\bu\|_{L^{\infty}} C_I}
                \right\}
        \end{equation}
hold true, then 
        \begin{equation}
                a_{\mathrm{SUPG}}(\tilde{u}, \tilde{u}) \geq \frac{1}{2} \| \tilde{u}\|^2_{\mathrm{SUPG}},
        \end{equation}
        where $\|\tilde{u}\|_{\mathrm{SUPG}}^2 =  \varepsilon \|\nabla \tilde{u}\|^2 + \delta \sum_{K \in \mathcal{T}_{\!h}} \|\bu \cdot \nabla \tilde{u} \|^2_{K,\Lthm} + \|c^{\nicefrac{1}{2}} \tilde{u}\|^2 $.
        This norm is suitable for advection-dominated problems, as it offers a better control of the stream-line diffusion. 
        As an immediate consequence of~\eqref{eqn:deltaK-coerc}, \( \|\vti + \delta \bu \cdot \nabla \vti \| \leq 2 \| \vti \|\). 
        Two additional properties of the SUPG setting are summarised below : 
        \begin{lemma} \label{lem:supg-bounds}
                Assuming~(\hyperlink{coefa}{\textsc{CoefA}}), it holds
                \begin{align}
                        \aSUPG(\tilde{u}, \vti) \leq C_1 \|\nabla \tilde{u} \| \| \vti  \|, && \| \tilde{u}\| \leq c_0^{-1} \|\tilde{u}\|_{\mathrm{SUPG}},
                \end{align}
                where $C_1 = (C_I + 2)\| \bu\|_{L^{\infty}} + 2 C_P \|c\|_{L^{\infty}_{\hmu}(L^{\infty})}$.  
        \end{lemma}
        \begin{proof}
                We detail the proof for some terms, the bounds for the others being direct. Firstly, $\varepsilon |(\nabla \tilde{u}, \nabla \vti )| \leq \|\nabla \tilde{u}\| \|\varepsilon \nabla \vti \| \leq \frac{C_I \|\bu\|_{L^{\infty}}}{2} \|\nabla \tilde{u}\| \|\vti \|$, having used $\varepsilon < \frac{1}{2} \| \bu \|_{L^{\infty}} h$ and the inverse inequality. Additionally, letting $C_2 = \frac{C_I }{2}\|\bu\|_{L^{\infty}}$,
                \begin{equation*}        
                         | \delta \sum_{K \in  \Th} (\varepsilon \Delta \tilde{u}, \bu \cdot \nabla \vti)_{K,\Lthm}| 
                        \leq 
                        C_2 \sum_{K \in \Th} \| \nabla \tilde{u} \|_{K,\Lthm} \| \vti  \|_{K,\Lthm}
                        \leq
                        C_2
                        \| \nabla \tilde{u}\| \| \vti\|. 
                \end{equation*}
        \end{proof}
        In~\cite{notr23}, we use Algorithm~\ref{alg:evaspl} reproduced below to sequentially update the physical and stochastic modes in a (potentially) cheap fashion, resulting in a non-linear update on~$\mathcal{M}_R$. The algorithm was originally proposed and analysed in~\cite{kavino21} for random uniform coercive problems, and is very similar to the Projector-Splitting algorithm~\cite{luos14}.
In this work, we focus on the implicit version of the scheme; however, semi-implicit and fully explicit versions are also possible. 
\begin{alg} \label{alg:evaspl}
Given the solution $\uhr[n] = \sum_{i=1}^R U_i^n Y_i^n$ :  
\begin{enumerate}
        \item Find $\tilde{U}^{n+1}_j, j= 1,\ldots,R$, such that
	\begin{multline}
          \dt^{-1}(\tilde{U}^{n+1}_j - U^n_j,v_{h} + \delta \bu \cdot \nabla v_{h})_{L^2(D)}
		+ \aSUPG(\uhr[n+1], v_h Y^n_j) 
    \\ 
    = (f^{n+1}, v_h Y^n_j + \bu \cdot \nabla v_h Y^n_j),  
    \quad \forall v_{h} \in V_h. \label{eqn:PG-disc-varf-detmodes-3}
	\end{multline}
\item Find $\tilde{Y}^{n+1}_j, j= 1, \ldots,R$ such that $(\tilde{Y}^{n+1}_j - Y^n_j) \in \mathcal{Y}^{\perp} = \mathcal{P}^{\perp}_{\mathcal{Y}}(\Lthm)$ and
		\begin{multline}
            \dt^{-1} \sum_{i=1}^R \mathbb{E}[(\tilde{Y}^{n+1}_i - Y^n_i)  z]\tilde{W}^{n+1}_{ij}
		+  \aSUPG(\uhr[n+1], \tilde{U}^{n+1}_j \mathcal{P}^{\perp}_{\mathcal{Y}} z) 
    \\
    = (f^{n+1}, \tilde{U}_j^{n+1}\mathcal{P}^{\perp}_{\mathcal{Y}}z + \delta \bu \nabla \tilde{U}_j^{n+1}\mathcal{P}^{\perp}_{\mathcal{Y}}z ),
    \quad \forall z \in \Lthm. \label{eqn:PG-disc-varf-stochmodes-3}
		\end{multline}
    where $\tilde{W}^{n+1}_{ij} = (\tilde{U}^{n+1}_i,  \tilde{U}^{n+1}_j + \delta \bu \nabla \tilde{U}^{n+1}_j)_{L^2(D)}$.
	\item Reorthonormalise $\tilde{Y}^{n+1}$ such that $\mathbb{E}[Y^{n+1}_i  Y^{n+1}_j ] = \delta_{ij}$ and modify $\{\tilde{U}^{n+1}_i\}_{i=1}^R$ such that $\sum_{i=1}^R \tilde{U}^{n+1}_i \tilde{Y}^{n+1}_i = \sum_{i=1}^R U^{n+1}_i Y^{n+1}_i$.
	\item The new solution is given by $\uhr[n+1] = \sum_{i=1}^R U^{n+1}_i Y^{n+1}_i$.
\end{enumerate}
\end{alg}
When applying Algorithm~\ref{alg:evaspl}, the update verifies a variational formulation (Proposition~\ref{prop:disc-varf}) which allows to analyse the scheme using variational methods and, among others, prove norm-stability of the scheme (Proposition~\ref{prop:stab}).

\begin{proposition}(from~\cite{notr23}) \label{prop:disc-varf} The numerical solution by Algorithm~\ref{alg:evaspl} satisfies
\begin{multline} \label{eqn:discrete-varf-pg-2}
        \frac{1}{\dt}(\uhr[n+1] - \uhr[n], v_{\hhmu} + \delta \bu \cdot \nabla v_{\hhmu})	 + \aSUPG(\uhr[n+1], v_{\hhmu}) = (f^{n+1}, v_{\hhmu} + \delta \bu \cdot \nabla v_{\hhmu}), \\
	\forall v_{\hhmu} \in \mathcal{T}_{\tilde{U}^{n+1}{(Y^{n})^\top}} \mathcal{M}_R.
\end{multline}
\end{proposition}
\begin{proposition} (from~\cite{notr23}) \label{prop:stab}
        Assuming $\delta$ verifies~\eqref{eqn:deltaK-coerc} and $\delta \leq \nicefrac{\dt}{4}$, then it holds for the numerical solution computed by Algorithm~\ref{alg:evaspl}
        \begin{equation}
                \|\uhr[N]\|^2 + \sum_{n = 1}^N \dt \|\uhr[n]\|_{\mathrm{SUPG}}^2 \leq \|\uhr[0]\|^2 + \dt\left(\frac{4}{c_0} + 4 \delta \right) \sum_{j=1}^N \|f^j\|^2. \label{eqn:stab-estim-dlr}
        \end{equation}
\end{proposition}

\section{Error estimate}

The idea of the SUPG method is to skew the test space by~$\Hcal = (I + \delta \bu \cdot \nabla)$. Its ajoint is given by~$\Hcal^{*} = I - \delta \bu \cdot \nabla$ thanks to the zero-divergence of $\bu$. 
Denote $\mathcal{P}_{\mathcal{H}^*} : V_h \otimes \Lthm \rightarrow \mathcal{T}_u \mathcal{M}_R$ the oblique projection on the tangent space: 
\begin{equation} \label{eqn:skew-proj}
	(\mathcal{P}_{\mathcal{H}^{*}} \tilde{u}, \mathcal{H}^* w) = (\tilde{u}, \mathcal{H}^{*} w) \quad \forall w \in \mathcal{T}_u \mathcal{M}_R.
\end{equation}
Its well-posedness is ensured by the coercivity of $(u,\Hcal^{*}u) = \|u\|^2$ on $V_h \otimes \Lthm$. Hereafter, we use the shorthand notation 
$\vti^{\perp} \coloneqq \mathcal{P}_{\mathcal{H}^*}^{\perp} \tilde{v} = \vti - \mathcal{P}_{\mathcal{H}^*} \tilde{v}$ for any $\vti \in V_h \otimes L^2_{\hat{\mu}}(\hat{\Omega})$.
By definition of the projection, 
\begin{equation} \label{eqn:proj-orth}
(\mathcal{P}_{\mathcal{H}^*}^{\perp} \tilde{v}, \Hcal^{*} w) = 0 \quad \forall w \in \mathcal{T}_u \mathcal{M}_R.
\end{equation}
A useful property of the oblique projection is the following : 
\begin{lemma} \label{lem:obl-proj-bound}
	\begin{equation}
		\|I - \mathcal{P}_{{\Hcal}^{*}}  \| = \|\mathcal{P}_{{\Hcal}^{*}}
		 \| \leq 3. 
	\end{equation}
\end{lemma}
\begin{proof}
	The first equality is a standard result of projectors~\cite{sz06}. 
	Consider the orthogonal projector $\Pi : V_h \otimes \Lthm \rightarrow \mathcal{T}_u \mathcal{M}_R$ verifying $(\Pi \tilde{u},  w) = (\tilde{u},  w)$ for $w \in \mathcal{T}_u \mathcal{M}_R$,
	we have
	\begin{multline}
		\|(\mathcal{P}_{\mathcal{H}^{*}} -  \Pi) \tilde{u}\|^2 =  ((\mathcal{P}_{\mathcal{H}^{*}} -  \Pi) \tilde{u}, (I - \delta \bu \nabla) (\mathcal{P}_{\mathcal{H}^{*}} -  \Pi) \tilde{u}) 
		\\ 
		= (\tilde{u} - \Pi\tilde{u}, (I - \delta \bu \nabla ) (\mathcal{P}_{\mathcal{H}^{*}} -  \Pi) \tilde{u} ) \leq 2 \|\Pi^{\perp} \tilde{u}\| \| (\mathcal{P}_{\mathcal{H}^{*}} -  \Pi) \tilde{u} \|  
	\end{multline}
	From there, we conclude $\|\mathcal{P}_{\mathcal{H}^{*}} \tilde{u} \| \leq \|(\mathcal{P}_{\mathcal{H}^{*}} - \Pi) \tilde{u} \| + \|\Pi \tilde{u}\| \leq 3 \|\tilde{u}\|$.
\end{proof}

We will make use of the following assumptions to analyse the convergence of the SUPG-DLR method. 
The first is the standard Model Error Assumption, particularised to the SUPG-context. It asks that the dynamics neglected by the DLR approximation is negligible. 
This is a standard assumption made to analyse the convergence of DLR approximations
~\cite{celu22,luki16,kolu07}. 
\begin{assumption} (Model Error Assumption) For $n = 0, \ldots, N-1$, let $\hat{u}^n = \tilde{U}^{n+1} Y^n$ be the ``intermediate'' point obtained by Algorithm~\ref{alg:evaspl}. For $\nu \ll 1$, it holds
\begin{equation} \label{eqn:model-error}
	|\aSUPG(\hat{u}^n, v^{\perp}_{\hhmu}) - (f, \Hcal v^{\perp}_{\hhmu})| \leq \nu \|\tilde{v}\|,  \quad \forall{\tilde{v}} \in V_h \otimes \Lthm, 
	\quad \mathrm{for}\; \nu \ll 1.
\end{equation}
\end{assumption}
The second is an assumption on the $H^1$-stability of the physical basis. 
\begin{assumption} \label{ass:lbi} (Local basis inverse inequality)
	Given the DLR iterates $\{\uhr[n] = \tilde{U}^{n} \tilde{Y}^{n}\}_{n=1}^{N}$ obtained via Algorithm~\ref{alg:evaspl}, and denoting $(\mathbb{S}_n)_{ij} = ( \nabla \tilde{U}^{n+1}_i, \nabla \tilde{U}^{n+1}_j)_{L^2(D)}$ and $(\mathbb{M}_n)_{ij} = (\tilde{U}^{n+1}_i, \tilde{U}^{n+1}_j)_{L^2(D)}$ the stiffness and mass matrices associated to the physical basis $\{\tilde{U}^n_i\}_{i=1}^R$,	there exists a constant $\Clocbasinv < \infty$ such that 
	\begin{equation}
		\max_{n = 0, \ldots, N}
		\left(
		\sup_{x \in \mathbb{R}^R}
		\frac{x^{\top} \mathbb{S}_n x}{x^{\top}\mathbb{M}_n x}\right)
		\leq \Clocbasinv.
	\end{equation}
\end{assumption}
\noindent The functions $(U^n_1, \ldots, U^n_R)$ are typically globally supported and display regularity, justifying a moderate value for $\Clocbasinv$. 
	Assumption~\ref{ass:lbi} implies that, for any $n \geq 0$,
	\begin{equation} \label{eqn:inv-ineq-subspace}
		\| \nabla \tilde{U}^{n} Z^{\top} \| \leq \Clocbasinv \|\tilde{U}^{n} Z^{\top} \| \quad \mathrm{for} \; Z  \in [\Lthm]^R.
	\end{equation}

	The elliptic projection operator $\pi : L^2_{\hmu}(\hat{\Omega}, H^1_0(D)) \rightarrow V_h \otimes \Lthm$ is defined by
\begin{equation} \label{eqn:elliptic-proj}
	(\nabla (u - \pi \, u), \nabla v_{h,\hmu}) = 0, \quad \forall v_{h,\mu} \in V_h \otimes \Lthm.
\end{equation}
For brevity, denote $\pi^{n} u = \pi u(t_{n})$. 
We split $\uhr[n] - u(t_{n}) = (\uhr[n] - \pi^{n} u) + (\pi^{n} u - u(t_{n})) = \tilde{e}^{n} + \eta^{n}$. 
The interpolation error $\eta^{n}$ is bounded using standard estimates which, assuming $u(t_n) \in L^2_{\hmu}(H^{k+1})$ for any $n$, yields 
(see e.g.~\cite{quva08})

\begin{equation} \label{eqn:interp-err}
\mathcal{E}^{N}(\eta) \coloneqq 
	\|\eta^{N}\|^2 +
	\frac{\dt}{4} \sum_{j=1}^{N} \|\eta^{j}\|_{\mathrm{SUPG}}^2 
\lesssim 
h^{2k+1}.
\end{equation}
For the other error term, Proposition~\ref{prop:disc-varf} allows to derive 
\begin{multline} \label{eqn:err-eqn}
	\dt^{-1}\lrp{\tilde{e}^{n+1} - \tilde{e}^n, \vti} 
	+ \aSUPG(\tilde{e}^{n+1}, \vti) 
	= 
	\aSUPG(u(t_{n+1}) - \piun[n+1], \vti)
	\\
	+ (\dot{u}(t_{n+1}) -  \dt^{-1} (\piun[n+1] - \piun[n]), \Hcal \vti)
	- \delta \dt^{-1} ( \tilde{e}^{n+1} - \tilde{e}^n, \bu \cdot \nabla \vti)
	+ \aSUPG(\hat{u}^n, \vti^{\perp}) 
	\\
	- (f^{n+1}, \Hcal \vti^{\perp})
	- \aSUPG(\hat{u}^n - \uhr[n+1], \vti^{\perp})
	+ \dt^{-1}\lrp{\uhr[n+1] - \uhr[n], \Hcal \vti^{\perp}}, \quad \forall \vti \in V_h \otimes \Lthm.
\end{multline}
Note that the last term in~\eqref{eqn:err-eqn} vanishes by~\eqref{eqn:proj-orth}. 
One last technical lemma is needed before presenting the main result:
\begin{lemma} \label{lem:bound-UdeltaY} Let $\tilde{\delta}Y^n \coloneqq \tilde{Y}^{n+1} - Y^n$. It holds
	\begin{equation*} \label{eqn:udeltay-ineq}
		\dt^{-1}		\|\tilde{U}^{n+1} \tilde{\delta} Y^n\|^2 = \aSUPG(\uhr[n+1],\tilde{U}^{n+1} \tilde{\delta} Y^n) + (f^{n+1}, \Hcal \tilde{U}^{n+1} \tilde{\delta} Y^n ).
	\end{equation*}
\end{lemma}
\begin{proof}
	Start from~\eqref{eqn:PG-disc-varf-stochmodes-3}. Using the definition of $\tilde{W}^{n+1}_{ij}$, we rewrite it as 
	\begin{multline*}
		\dt^{-1}	(\sum_{i=1}^R \tilde{U}^{n+1}_i\tilde{\delta}Y^n_j , \tilde{U}^{n+1}_j z_j + \delta \bu \cdot \nabla\tilde{U}^{n+1}_j z_j ) 
		= 
		\aSUPG(\uhr[n+1], \tilde{U}^{n+1}_j \mathcal{P}^{\perp}_{\mathcal{Y}} z_j) 
		\\
		+ (f^{n+1}, \tilde{U}^{n+1}_{j} \mathcal{P}^{\perp}_{\mathcal{Y}} z_j + \delta \bu \cdot \nabla\tilde{U}^{n+1}_j \mathcal{P}^{\perp}_{\mathcal{Y}} z_j   )
		\quad \mathrm{for} \, j \in 1, \ldots,R, \forall z_j \in \Lthm.
	\end{multline*}
	Set $z_j = \tilde{Y}^{n+1}_j - Y^n_j$, the result is obtained by summing over $j$ since $\tilde{Y}^{n+1} - Y^n \in (\mathcal{Y}^n)^{\perp}$ (the l.h.s. becomes
	$\|\tilde{U}^{n+1} \tilde{\delta} Y^n\|^2$ by zero-divergence of $\bu$).
\end{proof}

\begin{theorem} \label{thm:err-estim}
	Let $\bu \in (L^{\infty}(D))^d$ such that $\mathrm{div}\bu = 0$, $c \in L^{\infty}_{\hmu}(L^{\infty}(D))$ and assume the true solution verifies $u, \partial_t u \in L^{\infty}(0,T; L^\infty_{\hmu}(H^{k+1}(D)))$, $\partial_t^2 u \in L^2(0,T; L^{\infty}_{\hmu}(H^1))$. 
	Under~(\hyperlink{coefa}{\textsc{CoefA}}),~\eqref{eqn:model-error},~\eqref{eqn:deltaK-coerc} as well as~$\delta \leq \nicefrac{\dt}{4}$,
	the DLR iterates $\{\uhr[n]\}_{n=0}^{N}$ of Algorithm~\ref{alg:evaspl}	satisfy 
	\begin{multline}
		\|u(t_N) - \uhr[N]\| + \left(\sum_{i=1}^N \dt \|u(t_i) - \uhr[i]\|^2_{\mathrm{SUPG}}\right)^{\nicefrac{1}{2}}
		\\
	\lesssim h^{k+1} + \dt + \delta^{\nicefrac{1}{2}} h^{k}\ + \delta^{-\nicefrac{1}{2}}h^{k+1}  + \|\pi^0 u - \uhr[0]\| + \nu.
	\end{multline}
\end{theorem}
\begin{proof}
	The proof largely follows the structure of the proof in~\cite[pp. 10-12]{jovo11}. 
	Testing against $\tilde{e}^{n+1}$, the first two terms in the r.h.s of~\eqref{eqn:err-eqn} verify 
\begin{multline*}	
	\aSUPG(u(t_{n+1}) - \pi^{n+1} u, \tilde{e}^{n+1}) 
	+ (\dot{u}(t_{n+1}) - \dt^{-1} (\pi^{n+1} u - \pi^n u), \Hcal \tilde{e}^{n+1}) 
	\\
	= 
	\delta \sum_{K \in \Th} (\tilde{T}^{n+1}_{\mathrm{stab},K}, \bu \cdot \nabla \tilde{e}^{n+1})_{K,\Lthm}
	+ (T^{n+1}_{\mathrm{zero}}, \tilde{e}^{n+1}) 
	+ (T^{n+1}_{\mathrm{conv}}, \tilde{e}^{n+1}),
\end{multline*}
where 
\begin{align*}
	T^{n+1}_{\mathrm{zero}} &= (\dot{u}(t_{n+1}) - \pi^{n+1}\dot{u}) + c(u(t_{n+1}) - \pi^{n+1} u) + \left( \pi^{n+1} \dot{u} - \frac{\pi^{n+1}u - \pi^n u}{\dt} \right), \\
	T^{n+1}_{\mathrm{conv}} &= \bu \cdot \nabla (u(t_{n+1}) - \pi^{n+1} u), \\
	\tilde{T}^{n+1}_{\mathrm{stab},K} &= \left(T^{n+1}_{\mathrm{zero}} + T^{n+1}_{\mathrm{conv}} + \varepsilon \Delta (\pi^{n+1} u - u(t_{n+1})) \right)_{\rvert K}.
\end{align*}
Counter-integrating $(T^{n+1}_{\mathrm{conv}}, \tilde{e}^{n+1})$ and using the zero-divergence of $\bu$ yields
\begin{equation*}
	(T^{n+1}_{\mathrm{conv}}, \tilde{e}^{n+1}) = - \delta\sum_{K \in \Th}  \left( \delta^{-1}(\pi^{n+1} u - u(t_{n+1})) , \bu \cdot \nabla \tilde{e}^{n+1}\right)_{K,\Lthm},
\end{equation*}
which can then be included in $T^{n+1}_{\mathrm{stab},K}$, defining
\begin{equation*}
	T^{n+1}_{\mathrm{stab},K} = \tilde{T}^{n+1}_{\mathrm{stab},K} - \delta^{-1} (\pi^{n+1} u - u(t_{n+1})).
\end{equation*}
We then bound the terms \textit{via} Young's inequality, suitably balancing the coefficients such that the $\tilde{e}^{n+1}$-quantities on the r.h.s can be absorbed by $\frac{1}{2} \|e^{n+1}\|_{\mathrm{SUPG}}^2$ on the l.h.s. 
To this end, let $0 < \gamma \leq \nicefrac{1}{16}$.
As $(2\dt)^{-1} (\|\tilde{e}^{n+1}\|^2 - \|\tilde{e}^n\|^2 + \|\tilde{e}^{n+1} - e^n\|^2) + \nicefrac{1}{2} \|\tilde{e}^{n+1}\|^2_{\mathrm{SUPG}}$ lower-bounds the l.h.s of~\eqref{eqn:err-eqn}, it holds 
\begin{multline*}	
	(2\dt)^{-1} (\|\tilde{e}^{n+1}\|^2 - \|\tilde{e}^n\|^2 + \|\tilde{e}^{n+1} - \tilde{e}^n\|^2) + \nicefrac{1}{2} \|\tilde{e}^{n+1}\|^2_{\mathrm{SUPG}} 
	\\
	\leq 
	\delta \sum_{K \in \Th} (T^{n+1}_{\mathrm{stab},K}
	- 
	\dt^{-1} (\tilde{e}^{n+1} - \tilde{e}^n), \bu \cdot \nabla \tilde{e}^{n+1})_{K,\Lthm}
	+ (T^{n+1}_{\mathrm{zero}}, \tilde{e}^{n+1}) 
  \\
	+ \aSUPG(\uhr[n+1] - \hat{u}^n, (\tilde{e}^{n+1})^{\perp}) 
	+ \aSUPG(\hat{u}^n, (\eti{n+1})^{\perp}) 
	- (f^{n+1}, \Hcal (\eti{n+1})^{\perp})
	\\
	\leq 
 C	\delta \sum_{K \in \Th} \| T^{n+1}_{\mathrm{stab},K} \|_{K,\Lthm}^2
	+ \delta \gamma \| \bu \cdot \nabla \tilde{e}^{n+1}\|^2_{K,\Lthm}
	+ 
	C \dt^{-1} \| \tilde{e}^{n+1} - \tilde{e}^n \|^2 	+ C \| T^{n+1}_{\mathrm{zero}}\|^2 
	\\
	+  \gamma \|\tilde{e}^{n+1}\|^2
	+ \aSUPG(\uhr[n+1] - \hat{u}^n, (\eti{n+1})^{\perp}) 
	+ \aSUPG(\hat{u}^n, (\eti{n+1})^{\perp}) 
- (f^{n+1}, \Hcal (\eti{n+1})^{\perp}),
\end{multline*}
having used $\delta \lesssim \dt$ in the last inequality, and where $C$ depends on $\gamma^{-1}$. Lemma~\ref{lem:bound-UdeltaY} with~\eqref{eqn:inv-ineq-subspace} and~\eqref{eqn:model-error} respectively yield
\begin{align*}
	& \aSUPG(\tilde{U}^{n+1} \tilde{\delta} Y, (\eti{n+1})^{\perp}) 
	\lesssim \|\tilde{U}^{n+1} \tilde{\delta} Y\| \|\tilde{e}^{n+1} \| 
	\leq C \dt^2 (\|\uhr[n+1]\|^2 + \|f\|^2) + \gamma \|\tilde{e}^{n+1}\|^2, 
	\\
	& \aSUPG(\hat{u}^n, (\eti{n+1})^{\perp}) - (f^{n+1}, \Hcal (\eti{n+1})^{\perp}) 
	\leq \nu \| \tilde{e}^{n+1} \| 
	\leq C \nu^2 + \gamma \| \tilde{e}^{n+1} \|^2.
\end{align*}
Note that $\sum_{j=0}^{N-1} \dt^2 (\|\uhr[j]\|^2 + \|f\|^2) \lesssim \dt$ by Proposition~\ref{prop:stab}.
Cancelling, rearranging a few terms and summing over $j=0,\ldots,N-1$, we obtain
\begin{multline*}
	\|\tilde{e}^{n+1}\|^2 + \sum_{n=1}^N \|\tilde{e}^{n}\|_{\mathrm{SUPG}}^2 \lesssim \|\tilde{e}^0\|^2 + {\dt}\sum_{n=1}^N \|T_{\mathrm{zero}}^{n}\|^2 + \dt \sum_{n=1}^N \sum_{K \in \Th} \delta_K \| T^{n}_{\mathrm{stab},K}\|^2_K + \nu^2 + \dt^2. 
\end{multline*}
As in~\cite{jovo11}, the regularity assumptions on $u$ and its derivatives allow to bound
\begin{multline*}
\dt \sum_{n = 1}^N \|T^n_{\mathrm{zero}}\|^2 + \dt \sum_{n=1}^N \sum_{K \in \Th} \delta \|T^n_{\mathrm{stab},K}\|^2 \lesssim h^{2k + 2} + \dt^2 + \delta h^{2k} + h^{2k+2} \delta^{-1}.
\end{multline*}
Denoting $\gamma^n \coloneqq \uhr[n] - u(t_n)$, the claim follows as
$\mathcal{E}^N(\gamma) \lesssim \mathcal{E}^{N}(\tilde{e}) + \mathcal{E}^N(\eta)$. 


\end{proof}

\section{Numerical experiments}

We solve problem~\eqref{eqn:adv-diff-reac} on~$D = [0,1]$ with 
\begin{align*}
  \varepsilon = 10^{-8}, && \bu = 1, && c(x,\omega) = 1 + \omega, && \omega \sim \mathcal{U}[0,1]
\end{align*}
and choose the right-hand-side such that the true solution is given by
\begin{equation} \label{eqn:xp-true-sol}
  u_{\mathrm{true}}(t, x, \omega) = e^{x \sin(2 \pi \omega (t+1))} \sin(2\pi x). 
\end{equation}
The stochasticity therefore resides in the initial conditions, reaction term and forcing term. 
The sample space $\Omega = [0,1]$ is then approximated with the discrete set $\hat{\Omega} = \left\{ \nicefrac{i}{N_C} \right\}_{i=1}^{N_C}$ with $N_C = 15$ and equal probabilities in all the sample points. 
The physical space is discretised using a regular mesh with increasingly fine mesh size $h_i \sim 2^{-i}$. For the initial conditions, we compute a (generalised) SVD of~\eqref{eqn:xp-true-sol} at time $t=0$. 

As in~\cite{jovo11}, the terms $\dt$, $\delta^{\nicefrac{1}{2}}h^k$ and $h^{k+1} \delta^{-\nicefrac{1}{2}}$ in the error estimate need to be balanced to yield the best possible decay rate for a fixed $h$. Since Proposition~\ref{prop:stab} requires $\delta \sim \dt$, this imposes the condition $\dt \sim \mathcal{O}(h^{\frac{2(k+1)}{3}})$.

For the simulations we do not use the implicit scheme, but a semi-implicit version close to it that is both more technical and practical (reyling on a slightly different parametrisation of the approximation manifold with isolated mean, see~\cite{notr23}). With some technical details, the results carry over for that time-stepping scheme too.

In the first numerical experiment, we choose a rank $R = 6$ to ensure the error associated to the rank truncation is negligible. The rates observed in Figure~\ref{fig:supg-err-rank-6} are those predicted by Theorem~\ref{thm:err-estim}, both for the $L^2_{\hmu}(L^2(D))$ and the SUPG error. Figures~\ref{fig:supg-err-rank-deg-1} and~\ref{fig:supg-err-rank-deg-2} display the errors of DLR approximations computed with $R = 1,2,3$. The error is quasi-optimal with respect to the error obtained when using the optimal rank-$R$ truncation.

\begin{figure}
\centering
\begin{subfigure}{.32\textwidth}
  \centering
  \includegraphics[width=\linewidth]{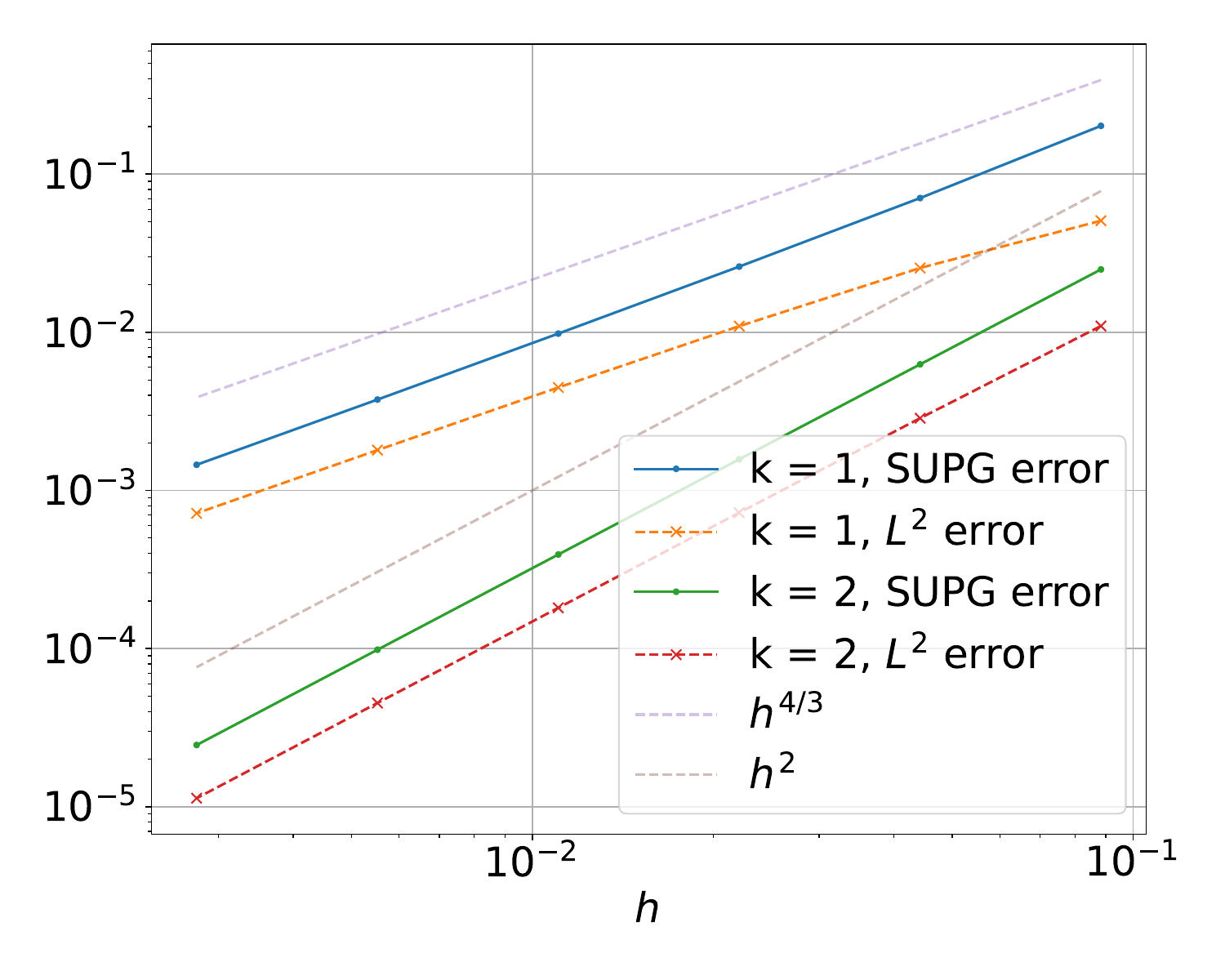}
  \caption{Rank $R = 6$}
  \label{fig:supg-err-rank-6}
\end{subfigure}%
\begin{subfigure}{.32\textwidth}
  \centering
  \includegraphics[width=\linewidth]{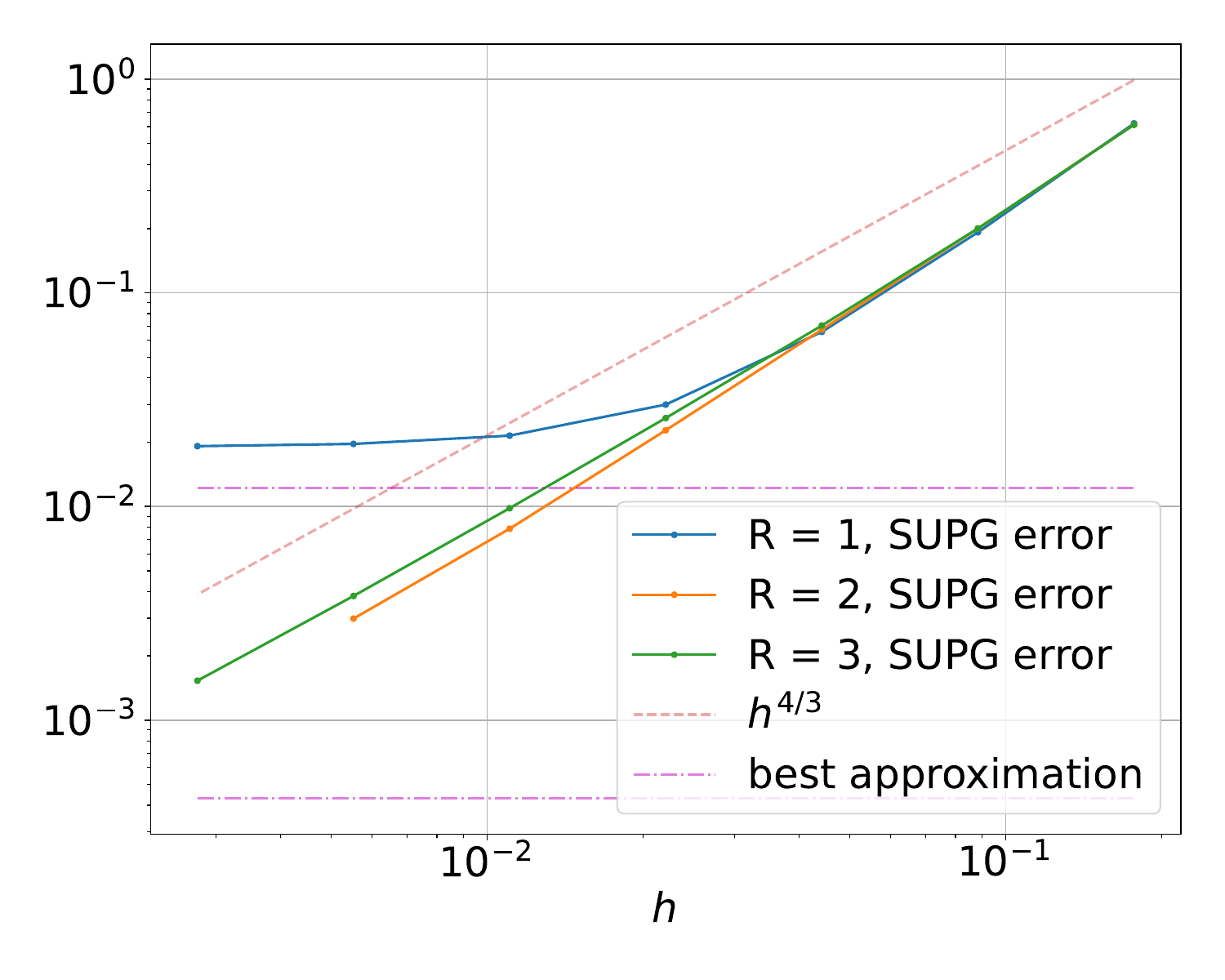}
  \caption{$V_h = \mathbb{P}^1(\Th)$ }
  \label{fig:supg-err-rank-deg-1}
\end{subfigure}%
\begin{subfigure}{.32\textwidth}
  \centering
  \includegraphics[width=\linewidth]{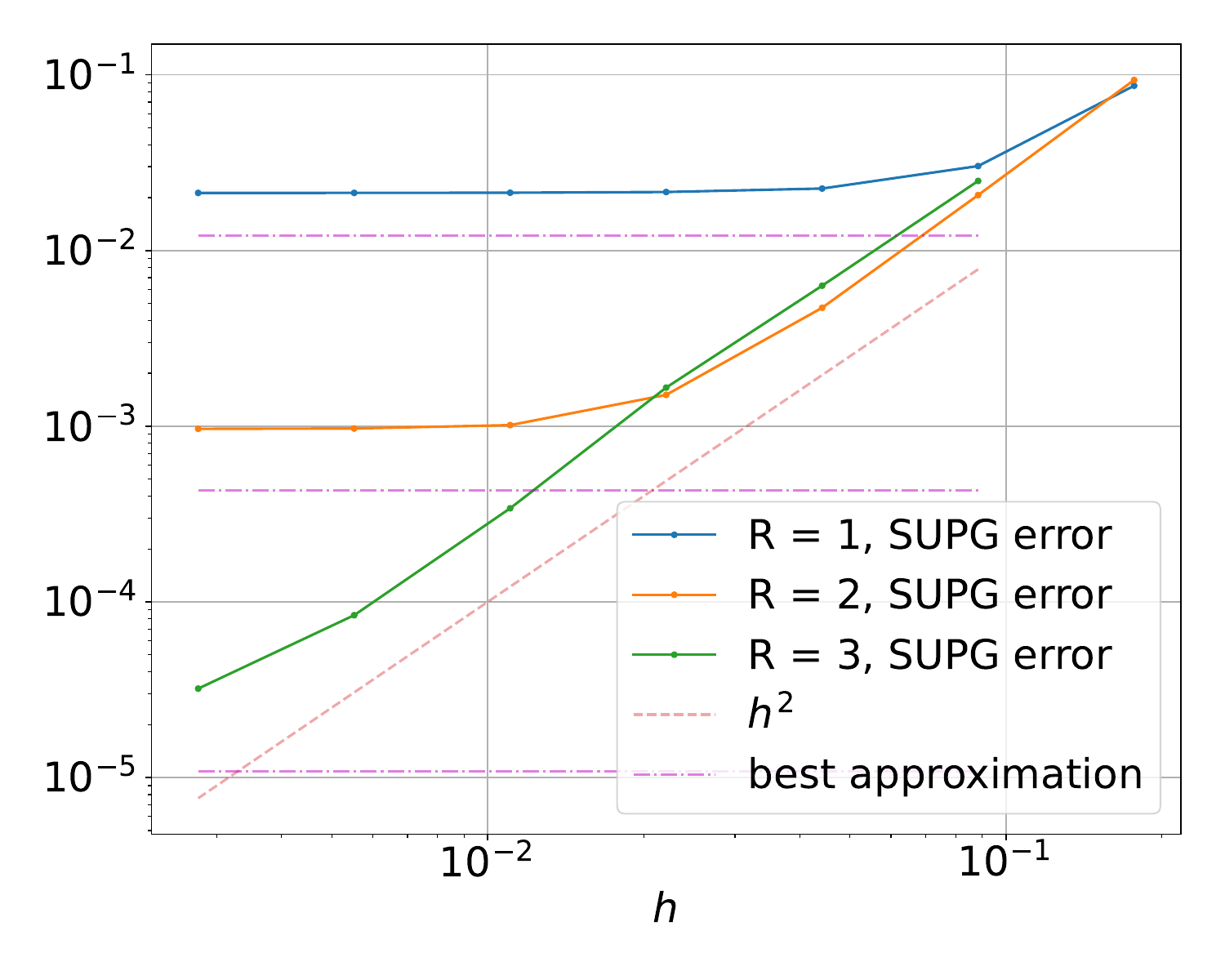}
  \caption{$V_h = \mathbb{P}^2(\Th)$}
  \label{fig:supg-err-rank-deg-2}
\end{subfigure}
\caption{SUPG error for $k = 1, 2$ and small approximation rank $R$.}
\label{fig:supg-err-rank}
\end{figure}

\vspace{-1em}

\bibliographystyle{siam} 
\bibliography{sdraft/refs.bib} 

\end{document}